\documentclass{amsart}
\usepackage{hyperref}
\usepackage{amssymb, amsmath, graphicx, mathrsfs, enumitem, aliascnt}

\newcommand{\Q}{\mathbb{Q}}

\newcommand{\Z}{\mathbb{Z}}
\newcommand{\C}{\mathbb{C}}
\newcommand{\F}{\mathbb{F}}

\newcommand{\SL}{\mathrm{SL}}

\newcommand{\PSL}{\mathrm{PSL}}
\newcommand{\Hom}{\mathrm{Hom}}

\usepackage{color}

\newtheorem{theorem}{Theorem}[section]

\newaliascnt{lemma}{theorem}
\newtheorem{lemma}[lemma]{Lemma}
\aliascntresetthe{lemma}

\newaliascnt{cor}{theorem}

\aliascntresetthe{cor}

\newaliascnt{prop}{theorem}

\aliascntresetthe{prop}

\newaliascnt{con}{theorem}

\aliascntresetthe{con}

\newaliascnt{defn}{theorem}
\newtheorem{defn}[defn]{Definition}
\aliascntresetthe{defn}

\theoremstyle{remark}
\newaliascnt{remark}{theorem}
\newtheorem{remark}[remark]{Remark}
\aliascntresetthe{remark}

\newaliascnt{claim}{remark}
\newtheorem{claim}[claim]{Claim}
\aliascntresetthe{claim}

\newaliascnt{question}{remark}
\newtheorem{question}[question]{Question}
\aliascntresetthe{question}

\def\sek~{\S{}}

\numberwithin{equation}{section}

\begin{document}

\title[Surfaces as intersections in the Character Variety]{Detecting essential surfaces as intersections in the Character Variety}

\author{Michelle Chu}
\address{Department of Mathematics, University of Texas at Austin, Austin, TX 78750, USA}
\curraddr{}
\email{mchu@math.utexas.edu}

\begin{abstract}
We describe a family of hyperbolic knots whose character variety contain exactly two distinct components of characters of irreducible representations. The intersection points between the components carry rich topological information. In particular, these points are non-integral and detect a Seifert surface.
\end{abstract}

\maketitle

\section{Introduction}
The $\SL_2\C$ character varieties of the fundamental groups of hyperbolic 3-manifolds carry a lot of topological information. In particular, in \cite{cullershalen}, Culler and Shalen developed a technique to detect embedded essential surfaces in a 3-manifold that arise from non-trivial actions of the fundamental group on a tree arising from ideal points in the $\SL_2\C$ character variety. The $\SL_2\C$ character variety of a hyperbolic knot group contains multiple components including the canonical component, which contains the character of a holonomy representation, and a component containing the characters of reducible representations. We address the following question:

\begin{question} \label{ques: 1}
How do multiple components in the $\SL_2\C$ character variety interact? In particular, what can we say about the characters in the intersection between multiple components?
\end{question}

In this paper we consider a family of two-bridge knots whose character varieties contain two distinct curves containing characters of irreducible representations. For this family, the existence of multiple curves was known to Ohtsuki in \cite{ohtsuki} and the existence of exactly two distinct curves was shown by Macasieb, Peterson, and van Luijk in \cite{smooth}. The main result of this paper is the following theorem.

\begin{theorem}\label{intro thm}
There exists infinitely many two-bridge knots having two distinct algebraic curve components of irreducible representations in their character varieties and whose intersections points detect a Seifert surface.
\end{theorem}

As is well known, components of the character varieties of two-bridge knots have the structure of algebraic curves which lie naturally in $\C P^2$ (see \autoref{subsec: char var std}). As such, Bezout's theorem guarantees finitely many points of intersection between any two curves. Some of these points are ideal, so that, following the methods in \cite{cullershalen}, they detect essential surfaces. It turns out that for this family, affine intersection points determine characters of algebraic non-integral, irreducible representations and also give interesting topological information. We once again obtain non-trivial actions on a tree, and hence also detect essential surfaces  (see \autoref{subsec: ani}).

Interestingly, the characters in the intersection are non-integral over the prime 2. In addition to these, one can check by explicit computation that the character varieties for the two-bridge knots $7_{7}$, $8_{11}$, $9_{6}$, $9_{10}$, $9_{17}$, $10_{5}$, $10_{9}$ and $10_{32}$ contain exactly two distinct curves of irreducible representations. It is also true in these examples that affine intersections between multiple curves are algebraic non-integral, correspond to irreducible representations, and furthermore, the trace of the meridian in these representations fails to be integral by a prime over 2. 

There appears to be no algebro-geometric reason as to why these affine intersection points are non-integral, and in particular non-integral by a prime over 2. For instance, computed examples of affine intersection points between curves of characters for two different knots were sometimes integral and other times not. This data suggests the following questions.

\begin{question}
Suppose $K$ is a hyperbolic two-bridge knot with multiple components of characters of irreducible representations in its character variety. When are intersection points between these components algebraic non-integral? When are they non-integral over the prime 2? What slopes are detected? What happens for general knots?
\end{question}

\subsection{Outline}
The paper is organized as follows. In \autoref{sec:prelim} we give the necessary background on character varieties, two-bridge knots, boundary slopes and actions on trees associated to algebraic non-integral representations. We also introduce the family of two-bridge knots of interest in this paper. \autoref{sec:models} builds on the work of Macasieb, Petersen and van Luijk in \cite{smooth}. We construct the character variety, define the smooth variety birationally equivalent to the character variety introduced in \cite{smooth}, and then use the birational equivalence to describe the points of intersection between components. In \autoref{sec:proofs} we state and prove a precise version of \autoref{intro thm} and determine the surfaces detected by affine intersection points. In \autoref{sec:examples} we describe in detail the intersection points for the first two knots in the family. In \autoref{sec:final} we make some final remarks on the existence of multiple components and describe examples of other two-bridge knots with multiple components.

\subsection{Acknowledgments}
The author thanks her advisor Alan W. Reid for his generous guidance, support, feedback and encouragement. The author also thanks Daryl Cooper, Cameron Gordon, Darren Long, Matthew Stover and Anh Tran for insightful conversations and suggestions. Finally, we thank the anonymous referee for their comments and in particular for pointing out that the knots in question do not have a unique Seifert surface and for suggesting \autoref{lem: slope 0 is Seifert}.

\section{Preliminaries}\label{sec:prelim}

\subsection{Character varieties}\label{subse:charvarintro}
We begin with some background on representation varieties and in particular character varieties. For more on this material see \cite{cullershalen}. 

Let $\Gamma$ be a finitely generated group. The $\SL_2\C$-representation variety of $\Gamma$ is the set $\mathrm{R}(\Gamma)=\Hom(\Gamma,\SL_2\C)$ and has the structure of an affine algebraic set over $\Q$ with coordinates given by the matrix entries of the images of the the generators of $\Gamma$.

The $\SL_2\C$ character variety of $\Gamma$ is the set $\widetilde{X}(\Gamma)=\{\chi_\rho:\rho\in\mathrm{R}(\Gamma)\}$
where the character $\chi_\rho:\Gamma\rightarrow\C$ is the map defined by $\chi_\rho(\gamma)=tr(\rho\gamma)$ for all $\gamma\in\Gamma$.
For all $\gamma\in\Gamma$ define the map $t_\gamma:\mathrm{R}(\Gamma)\rightarrow\C$ by $t_\gamma(\rho)=\chi_\rho(\gamma)$. The ring $\mathrm{R}$ generated by 1 and the maps $t_\gamma$ for $\gamma\in\Gamma$ turns out to be finitely generated by, say, $\{  t_{\gamma_1},\dots,t_{\gamma_m} \}$ for some elements $\gamma_1,\dots,\gamma_m \in\Gamma$. It follows that a character $\chi_\rho\in\widetilde{X}(\Gamma)$ is determined by its values on the finitely many elements $\gamma_1,\dots,\gamma_m\in\Gamma$. We get that $\widetilde{X}(\Gamma)$ has the structure of a affine algebraic set in $\C^m$ with coordinate ring $R$. Different sets of generators for $\mathrm{R}$ give different models for $\widetilde{X}(\Gamma)$ which are all isomorphic over $\Z$.

An $\SL_2\C$ representation $\rho\in\mathrm{R}(\Gamma)$ is reducible if, up to conjugation, $\rho(\gamma)$ is upper triangular for every $\gamma$, and otherwise irreducible. An $\SL_2\C$ representation $\rho\in\mathrm{R}(\Gamma)$ is abelian if its image is abelian, and otherwise nonabelian. Every irreducible representation is nonabelian. However, there exist reducible nonabelian representations. The set of characters of abelian representations $X_{ab}(\Gamma)$ is itself a variety. Let $X_{na}(\Gamma)$ be the Zariski closure of $\widetilde{X}(\Gamma)-X_{ab}(\Gamma)$ and denote it by $X(\Gamma)$. 

If two representations $\rho,\rho'\in\mathrm{R}(\Gamma)$ are conjugate, then $\chi_\rho=\chi_{\rho'}$. Also if $\chi_\rho=\chi_{\rho'}$ and $\rho$ is irreducible then $\rho$ and $\rho'$ are conjugate. Therefore, when considering irreducible representations, we may think of $X(\Gamma)$ as the space of irreducible representations modulo conjugation.

Whenever $\Gamma$ is the fundamental group of an orientable, complete hyperbolic 3-manifolds of finite volume, there is an irreducible component of $X(\Gamma)$ containing the character of a holonomy representation of the 3-manifold. This component is called the {\it canonical component}.

One can also define the $\PSL_2\C$-character variety (see \cite[\S 3]{seminorms} \cite[\S 2.1]{longreid},\cite[\S 2.1.2]{smooth}). In the case of $\Gamma$ a knot group, the $\PSL_2\C$-character variety $\widetilde{Y}(\Gamma)$ is the quotient $\widetilde{X}(\Gamma)/\Hom(\Gamma,\pm1)$ where $\pm1$ is the kernel of the homomorphism $\SL_2\C\rightarrow\PSL_2\C$. It has as coordinate ring the subring of $R$ of elements invariant under $\pm1$. Let $Y(\Gamma)$ denote the image of $X(\Gamma)$ in $\widetilde{Y}(\Gamma)$. 

\subsection{Two-Bridge Knots}\label{subsec:2 bridge}
Two-bridge knots are those non-trivial knots admitting a knot diagram with two maxima. Every two-bridge knot is associated to a two-bridge normal form $(p,q)$ where $p$ and $q$ are integers with $p$ odd and $0<q<p$. Whenever $q\neq 1$, the associated knot is hyperbolic. Two knots with two-bridge normal forms $(p,q)$ and $(p',q')$ are equivalent if and only if $p=p'$ and either $q=q'$ or $qq'\equiv\pm1\mod p$.

The knot group corresponding to the two-bridge normal form $(p,q)$ has a presentation $\langle a,b : aw=wb\rangle$ where $a,b$ are meridians and $w=a^{\epsilon_1}b^{\epsilon_2} \cdots a^{\epsilon_{p-2}}a^{\epsilon_{p-1}}$ with $\epsilon_i=(-1)^{\lfloor iq/p \rfloor}$ and $\lfloor\cdot\rfloor$ the floor function (see \cite[Proposition 1]{riley3},\cite[Proposition 1]{mayland}). For this presentation the preferred meridian is given by $a$ and the corresponding preferred longitude is given by $ww^*a^{-2e(w)}$ where $w^*$ is $w$ written backwards and $e(w)=\sum\epsilon_i$ so that the total exponent sum of the longitude is 0 (see \cite[\S 2]{apoly}).

\subsection{A Family of Two-Bridge Knots}\label{subsec:family}

\begin{figure}
\begin{minipage}{.49\textwidth}
\centering
\includegraphics[width=92pt]{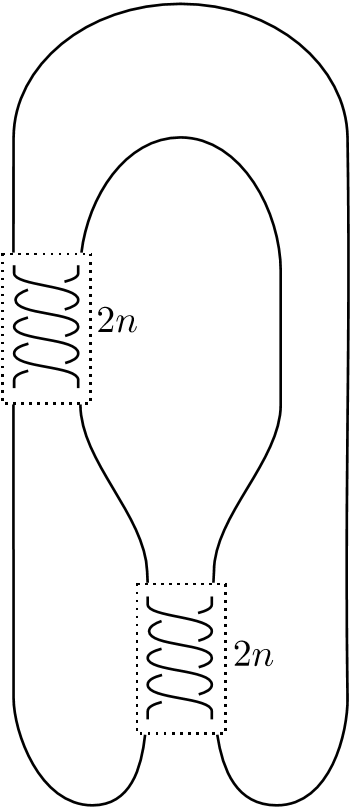}
\caption{The two-bridge knot $J(2n,2n)$}\label{fig:knots}
\end{minipage}
\begin{minipage}{.49\textwidth}
\centering
\includegraphics[width=100pt]{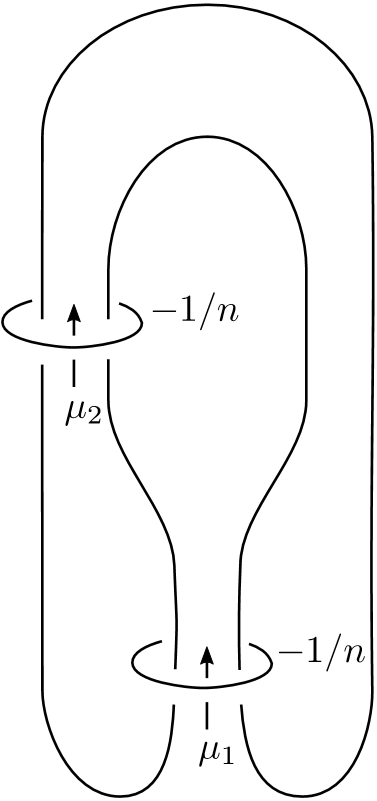}
\caption{The Borromean rings}\label{fig:ringssurg}
\end{minipage}
\end{figure}

The knots to be considered in this paper are the members of the family of hyperbolic two-bridge knots $J(2n,2n)$, $n\geq 2$, with two-bridge normal form $(4n^2-1, 4n^2-2n-1)$. Note that this form is equivalent to $(4n^2-1, 2n)$.  These have knot diagrams as shown in \autoref{fig:knots} and are obtained as  $-\frac{1}{n}$ and $-\frac{1}{n}$ surgeries on two components of the Borromean rings as in \autoref{fig:ringssurg}. The first knot in this family, $J(4,4)$, is the knot $7_4$ in the knot tables with two-bridge normal form $(15,11)$.

The knot group $\Gamma_n$ for $J(2n,2n)$ can be computed as in \cite[Proposition 1]{apoly}. It has presentation
\begin{equation}\label{presentation}
 \Gamma_n=\pi_1(S^3\setminus J(2n,2n)) =  \langle a,b : aw^n=w^nb\rangle
\end{equation}
where $w=(ab^{-1})^n(a^{-1}b)^n $. As in \autoref{subsec:2 bridge}, the preferred meridian is $a$ with corresponding preferred longitude $(w^n)(w^n)^*$.

\begin{figure}
\centering
\includegraphics[width=200pt]{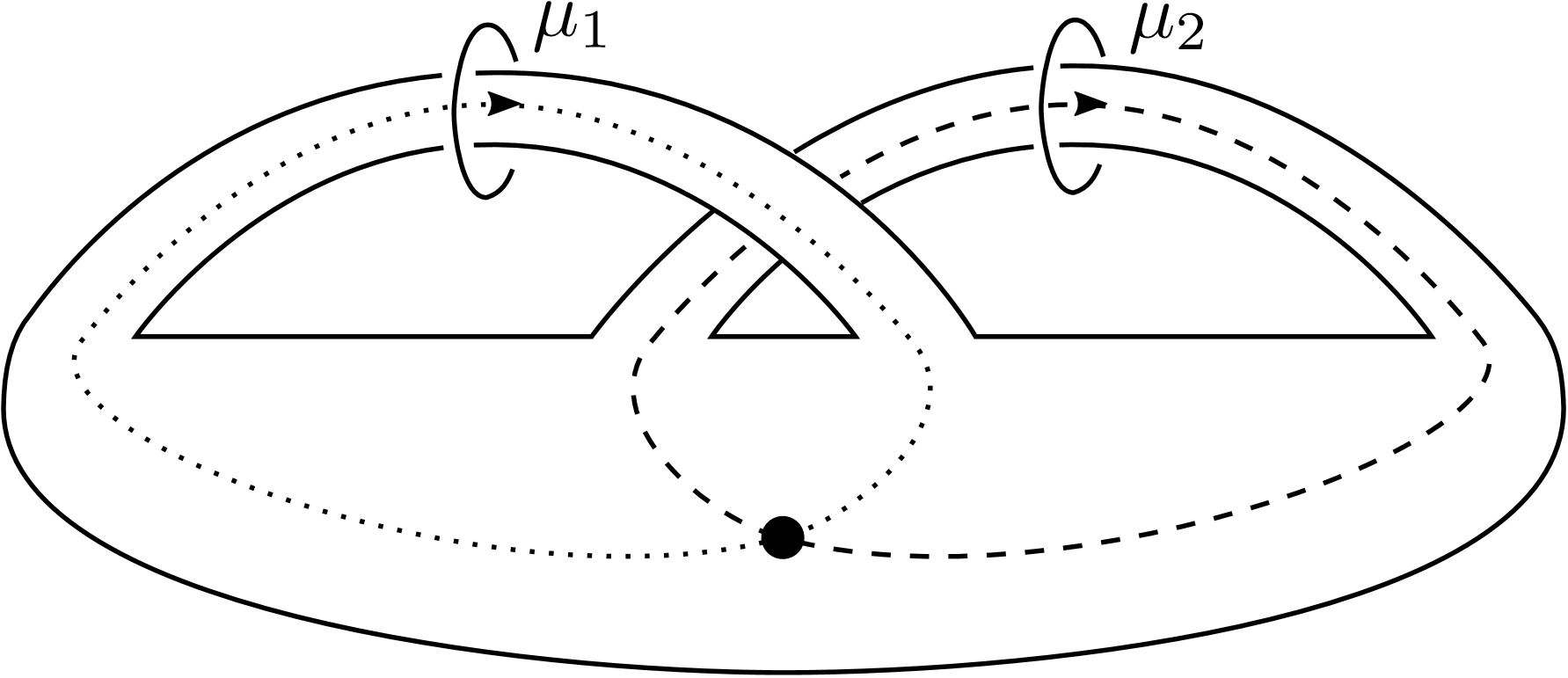}
\caption{Borromean rings after isotopy}\label{fig:ringsiso}
\end{figure}

These knots have an orientable Seifert surface of genus 1 whose fundamental group is generated by the images of the meridians $\mu_1$ and $\mu_2$ after the two Dehn surgeries (see \autoref{fig:ringsiso} and \autoref{fig:knotseifert}). From the proof of \cite[Proposition 1]{apoly}, these correspond to
\begin{equation}\label{eq: seifert gens}
s_1=\left((a b^{-1})^n(a^{-1} b)^n\right)^n \text{ and } s_2=(a b^{-1})^n .
\end{equation}
These two elements generate a free subgroup and their commutator $s_1 s_2^{-1} s_1^{-1} s_2$ corresponds to the preferred longitude. We note that this is not a unique Seifert surface. In fact, it can be shown following \cite{hatcherthurston} that there are two non-isotopic Seifert surfaces.

\begin{figure}
\centering
\includegraphics[width=200pt]{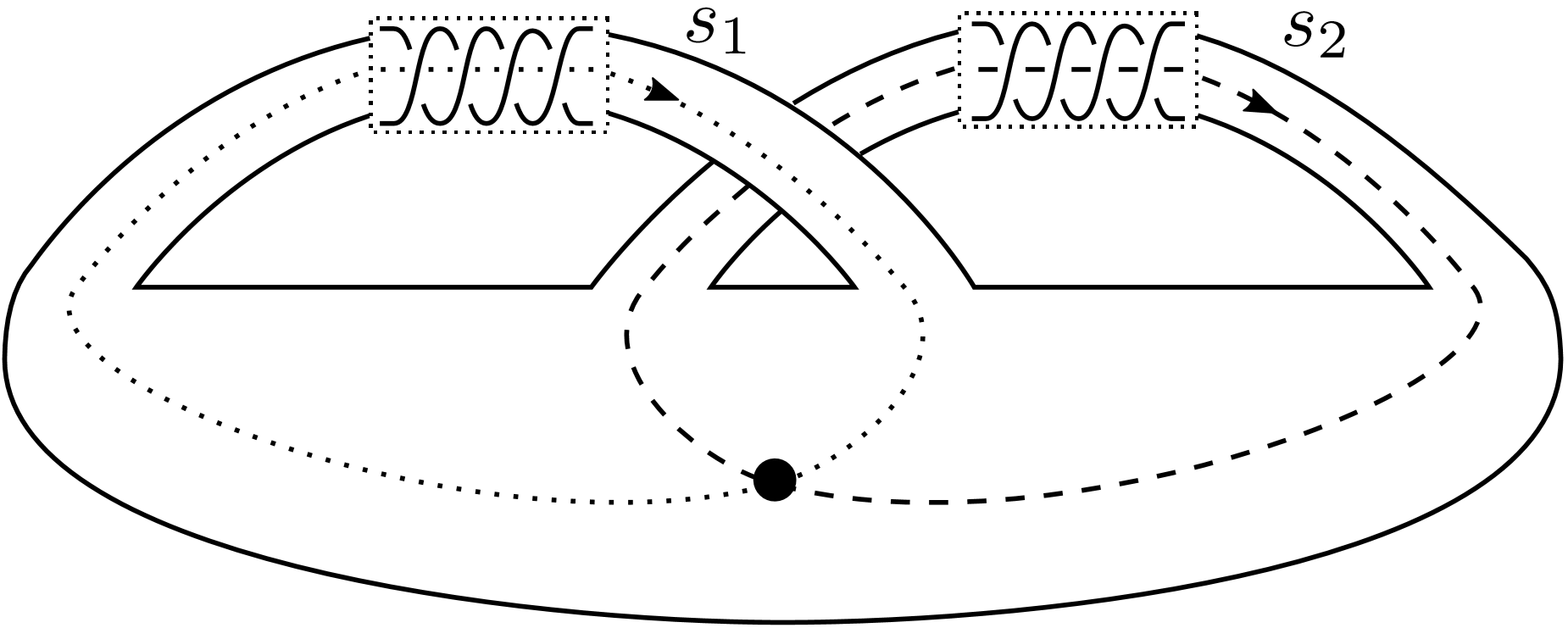}
\caption{A Seifert surface for the knot $J(2n,2n)$}\label{fig:knotseifert}
\end{figure}

\subsection{Boundary Slopes} \label{subsec:bslope} 
An essential surface in a 3-manifold is a properly embedded orientable incompressible surface which is not boundary parallel. Let $V:=V(K)$ denote the exterior of the knot $K$. Any embedded essential surface $S$ with non-empty boundary in $V$ will have non-empty boundary $\partial S=S\cap\partial V$, a collection of disjoint circles on the torus boundary of $V$. Since these circles are disjoint, they represent the same element in the fundamental group of the boundary torus. We identify $\pi_1(\partial V)$ with the group $\Z\times\Z$ where the factors are generated by the preferred meridian and the preferred longitude. Therefore, these circles in $\partial S$ correspond to a class $(p,q)\in\Z\times\Z$ where $p$ and $q$ are relatively prime. We call $p/q$ the slope of $S$, and represent it in $\pi_1(V)$ by the element $M^p L^q$ where $M$ is the meridian and $L$ the longitude. We say that $p/q$ is a boundary slope for $K$ if there is an essential surface $S$ in $V$ with slope $p/q$. We call the class $(0,1)$ the 0-slope and the class $(1,0)$ the $\infty$-slope.

Note that two-bridge knots have small exteriors, that is, they do not contain closed embedded essential surfaces (see \cite{hatcherthurston}).

\subsection{Algebraic Non-Integral Representations and actions on the tree} \label{subsec: ani}
The topics in this section can be found in \cite[\S II.1]{serre}, \cite[\S1,\S2.3]{cullershalen}, \cite[\S3]{handbook} and \cite[\S1]{schanuelzhang}. The description of the tree and the action follows from \cite[\S3]{handbook}.

Let $H$ be a number field with a discrete valuation $v:H^*\rightarrow \mathbb{\Z}\cup \{\infty\}$. There is a canonical way to construct a simplicial tree $T_{H,v}$ on which $\SL_2(H)$ acts without inversion. This construction was described by Serre in this form, but was previously discovered by Bruhat and Tits. Let $\mathcal{O}_v$ be the valuation ring and let $\pi$ be a choice of uniformiser. Define the graph $T_{H,v}$ with vertices given by the homothety classes of lattices in $H^2$ and an edge between two vertices if there exists representative lattices $\Lambda_0$ and $\Lambda_1$ and a linear automorphism $M$ of $H^2$ of determinant $\beta$ with $v(\beta)=1$ which maps $\Lambda_0$ onto $\Lambda_1$. It turns out that $T_{H,v}$ is a tree and $\SL_2(H)$ acts on it simplicially and without inversions by the action induced from the action on $H^2$.

When the fundamental group $\Gamma=\pi_1(V)$ has a representation $\rho$ into the group $\SL_2(H)$, there is an induced action of $\Gamma$ on the tree $T_{H,v}$ via the representation $\rho$. If the action of $\Gamma$ on $T_{H,v}$ is nontrivial, it induces a splitting of $\Gamma=\pi_1(V)$ along an edge stabilizer. By Culler and Shalen, there is an essential surface associated to this action (see \cite[Theorems 2.2.1,2.3.1]{cullershalen}). The fundamental group of this associated essential surface is contained in an edge stabilizer. We say such an essential surface is detected by the representation $\rho$.

Whenever there is an element $\gamma\in\Gamma$ with $v(tr(\rho\gamma))<0$, the action  of $\Gamma$ on $T_{H,v}$ is nontrivial. In particular, consider a representation $\rho: \Gamma\rightarrow \text{SL}_2(H)$ where $H$ is an algebraic number field and such that there is some element $\gamma\in \Gamma$ with $tr(\rho\gamma)$ not an algebraic integer. We call such a representation an {\it algebraic non-integral representation}. Then there is some prime ideal $\mathcal{P}$ in $\mathcal{O}_v$ such that $v_\mathcal{P}(tr(\rho(\gamma)))<0$. 

The following lemma is a restatement of Corollary 3 of \cite{schanuelzhang}. It describes how to determine the slope detected by a representation with an algebraic non-integral character.

\begin{lemma}\label{lem: schanuelzhang}
Let $V$ be a hyperbolic knot exterior and $\rho:\pi_1(V)\rightarrow\SL_2(k)$ a representation where $k$ is a number field. If $\chi_\rho(\gamma)$ is not an algebraic integer for some slope $\gamma\in\partial V$ but $\chi_\rho(\delta)$ is an algebraic integer for another slope $\delta\in\pi_1(\partial V)$, then
$\chi_\rho$ detects an essential surface $S$ in $V$ with boundary slope $\delta$. 
\end{lemma}

\section{Character Variety of $J(2n,2n)$}\label{sec:models}
\subsection{Character Varieties: The Standard Model}\label{subsec: char var std}
Consider the knot $J(2n,2n)$ with knot group presentation as in \autoref{presentation}. The generators $a$ and $b$ are conjugate in the group, so $t_a=t_b=t_{b^{-1}}$. We may take $t_a$ and $t_{ab^{-1}}$ as the generators for the ring $R$ defined in \autoref{subse:charvarintro} and also as coordinates in $X(\Gamma_n)$ (see \cite[Proposition 1.4.1]{cullershalen}).

A nonabelian representation $\rho_0\in R(\Gamma_n)$ with $t_a(\rho_0)=x$ and $t_{ab^{-1}}(\rho_0)=r$ is conjugate in $SL_2\C$ to a representation $\rho$ with $A=:\rho(a)$ and $B=:\rho(b)$ given by
\begin{equation}
 A= \begin{pmatrix} \mu & 1 \\ 0 & \mu^{-1} \end{pmatrix}
 \hspace{8pt}\text{and}\hspace{8pt}
 B= \begin{pmatrix} \mu & 0 \\ 2-r & \mu^{-1} \end{pmatrix}.
\end{equation}
This $\rho_0$ is reducible exactly when $r=2$.

We set $x=tr(A)$ and $r=tr(AB^{-1})$. Choose $\mu$ for which $x=\mu+\mu^{-1}$ and let $W_n=(AB^{-1})^n(A^{-1}B)^n$. Then an assignment of $\mu$ and $r$ extends to a representation if and only if $AW_n^n=W_n^nB$. This results in four equations on $\mu$ and $r$, one for each matrix coordinate. However, the vanishing set of these four equations can be defined by a single equation in $x$ and $r$ which is independent of the choice between $\mu$ and $\mu^{-1}$ (see \cite[Theorem 1]{riley2}).

\begin{defn}
Let $f_0(u)=0$, $f_1(u)=1$ and define $f_{j+1}(u) = u\cdot f_j(u)-f_{j-1}(u)$. Define also $g_j(u)=f_j(u)-f_{j-1}(u)$.
\end{defn}

The variety $X(\Gamma_n)$ is defined as the vanishing set of the polynomial
\begin{equation}\label{X}
f_n(t)\left( f_n(r)g_n(r)(-x^2+2+r) -1 \right)+f_{n-1}(t)
\end{equation}
where 
\begin{equation}\label{t}
t=tr(W_n)=(2-r)(x^2-2-r)f_n^2(r)+2
\end{equation}
and $W_n=(AB^{-1})^n(A^{-1}B)^n$ (see \cite[Proposition 3.8]{smooth}).

The variety $X(\Gamma_n)$ is an affine algebraic curve. Also, it is the double cover of the variety $Y(\Gamma_n)$ with variables $(r,y)$ via the covering map
\begin{align}\label{eq: covering map}
X(\Gamma_n) & \rightarrow Y(\Gamma_n)\\
(r,x) & \mapsto (r,x^2-2) \notag
\end{align}
(see \cite[\S 2.2.2]{smooth}).

Affine algebraic curves may be completed naturally into projective curves by homogenizing their defining polynomials. Therefore, we may think of $X(\Gamma_n)$ and $Y(\Gamma_n)$ as projective curves in $\C P^2$ composed of an affine part and finitely many points of completion, that is, ideal points.

\subsection{Character Varieties: The Smooth Model}\label{subsec: char var smooth}

The varieties $X(\Gamma_n)$ and $Y(\Gamma_n)$ for $J(2n,2n)$ are not smooth at infinity. To get around this, a new projective model $D(\Gamma_n)$ was introduced in \cite{smooth}. This new model is birationally equivalent to $Y(\Gamma_n)$ and each of its irreducible components is smooth.

Let $D(\Gamma_n)$ be the projective closure of the affine variety in the coordinates $r=t_{ab^{-1}}$ and $t=t_w$. It is the vanishing set of the polynomial
\begin{equation}\label{smoothpoly}
g_{n+1}(r)g_n(t)-g_n(r)g_{n+1}(t).
\end{equation}

\begin{theorem} \label{thm:smooth}
For $n\geq2$ in $\Z$ the following statements hold.
\begin{enumerate}[label=(\arabic*),ref=(\arabic*)]

\item $D(\Gamma_n)$ is birationally equivalent to $Y(\Gamma_n)$ via the map 
	\begin{equation}\label{bieq}
	Y(\Gamma_n)\rightarrow D(\Gamma_n) \text{ given by } (r,y)\mapsto (r,(2-r)(y-r)f_n^2(r)+2).
	\end{equation}

\item\label{r=2} $Y(\Gamma_n)$ and $D(\Gamma_n)$ are isomorphic outside a finite number of points $(r,y)$ in $Y(\Gamma_n)$ given by $(r-2)f_n(r)=0$.

\item\label{line r=t} $D(\Gamma_n)$ consists of two irreducible components: the component $D_0$ defined by the line $r-t$ and the component  $D_1$ defined as the projective closure of the complement of $D_0$. Furthermore, each component is smooth. 

\item $Y(\Gamma_n)$ has two irreducible components: the canonical component $Y_0$ and the component $Y_1$. Furthermore $Y_0$ is birationally equivalent to $D_0$ and $Y_1$ to $D_1$.

\item $X(\Gamma_n)$ has two irreducible components: the canonical component $X_0$ and the component $X_1$. Furthermore, $X_0$ is the double cover of $Y_0$ and $X_1$ the double cover of $Y_1$ (see \autoref{eq: covering map}).

\end{enumerate}
\end{theorem}

\begin{proof}
These statements are given in Proposition 4.4 and Proposition 4.6 of \cite{smooth}.
\end{proof}

\subsection{Intersections between components} \label{intersections section}
In this section we describe a polynomial $G_n$ which determines the $r$ coordinate of the intersection points of $D_0$ and $D_1$. We show $G_n$ also determines the $r$ coordinate in the affine intersection points of $Y_0$ and $Y_1$ via the birational equivalence, and thus also for the affine intersection points of $X_0$ and $X_1$.

\begin{defn}
Let $g'_{i}=\frac{dg_{i}}{du}$ and define 
\begin{equation}
G_j= g_{j+1}'g_j-g_{j+1}g_j'.
\end{equation}
\end{defn}

\begin{lemma}\label{fhfacts}
For $j\geq2$ in $\Z$ the following statements hold.
\begin{enumerate}[label=(\arabic*),ref=(\arabic*)]
\item\label{fpoly} $f_j$ is monic, separable and of degree $n-1$.
\item\label{hpol} $(u+2) G_j = f_{2j} + 2j$.
\item\label{hpoly} $G_j$ is monic and of degree $2j-2$. 
\item\label{hfnoroot} $G_j$ and $f_j$ do not share a root. 
\item\label{f2j} $f_{2j}=uf_{j}^2-2f_jf_{j-1}$.
\item\label{fj2} $f_j(2)=j$.
\end{enumerate}
\end{lemma}

\begin{proof}
Proofs for \ref{fpoly}, \ref{hpol} and \ref{hpoly} are found in \cite[Lemma 3.3, Lemma 5.4]{smooth} but we gather the necessary information here for convenience.
Recall that $f_0=0$, $f_1=1$ and $f_j=u \cdot f_{j-1}-f_{j-2}$. It follows by induction on $j$ that $f_j$ is monic and of degree $j-1$.
To prove separable consider the ring $\Z[u][s]/(s^2-us+1)\cong \Z[s,s^{-1}]$. Here $u=s+s^{-1}$ and it follows that $f_j(u)=(s^j-s^{-j})/(s-s^{-1})$ by induction on $j$. Consider the following polynomial in  $\Z[u][s]/(s^2-us+1)$:
\begin{align*}
 (s^{j+1}-s^{j-1})f_j & = (s^{j+1}-s^{j-1})\frac{(s^j-s^{-j})}{(s^1-s^{-1})} \\
 & = (s^{2j}-1)\frac{(s^1-s^{-1})}{(s^1-s^{-1})} \\
 & = s^{2j}-1. 
\end{align*}
Since $s^{2j}-1$ is separable, so is $f_j$ and this proves \ref{fpoly}.

The identity in \ref{hpol} follows directly from the definition of $G_j$ by considering its image in the ring $\Z[u][s]/(s^2-us+1)$.

By \ref{fpoly}, $f_{2j}$ is monic and of degree $2j-1\geq 3$. Using the identity in \ref{hpol}, $G_j$ is also monic and is of degree $(2j-1)-1=2j-2$, proving \ref{hpoly}.

Suppose $\omega$ is a root of $G_j$ and pick $\sigma\in\C^*$ such that $\omega=\sigma+\sigma^{-1}$. By \ref{hpol}, $0=f_{2j}(\omega)+2j$ so then
\begin{align*}
-2j & = f_{2j}(\omega) \\
& = \frac{(\sigma^{2j}-\sigma^{-2j})}{(\sigma-\sigma^{-1})}\\
& = (\sigma^{j}+\sigma^{-j}) \frac{(\sigma^{j}-\sigma^{-j})}{(\sigma-\sigma^{-1})}\\
& =(\sigma^{j}+\sigma^{-j})f_j(\omega).
\end{align*}
Then \ref{hfnoroot} follows since this implies $f_j(\omega)\neq 0$.

The identity in part \ref{f2j} follows by considering $f_j(u)=(s^j-s^{-j})/(s+s^{-1})$ in the ring $\Z[u][s]/(s^2-us+1)$ to get
$$ f_{2j}=f_{j}(f_{j+1}-f_{j-1})$$ and using the recursive definition for $f_{j+1}$.

The identity in part \ref{fj2} follows from induction on $j$.

\end{proof}

We can now describe the points of intersection between $D_0$ and $D_1$.

\begin{lemma}\label{rcoor}
If the point $P=(r_0,t_0)$ is in the intersection of $D_0(\Gamma_n)$ and $D_1(\Gamma_n)$, then $P$ satisfies $G_n(r_0)=G_n(t_0)=0$.
\end{lemma}

\begin{proof}
A similar statement is included in \cite[Lemma 5.5]{smooth}. We include a complete proof for the relevant case.

In the ring $\Z[u][s]/(s^2-us+1)$,
$$ f_n^2 - f_{n-1}f_{n+1} = \frac{(s^n-s^{-n})^2}{(s-s^{-1})^2} - \frac{(s^{n-1}-s^{1-n})(s^{n+1}-s^{-n-1})}{(s-s^{-1})^2} = 1.$$
As a polynomial, we have that $f_n^2-f_{n-1}f_{n+1}=1$.

Recall $F=g_{n+1}(r)g_n(t)-g_n(r)g_{n+1}(t)$ is the defining polynomial for $D(\Gamma_n)$. Since $P$ is a point in the intersection of two components, it is a singular point of $D(\Gamma_n)$. Therefore $F_r:=\partial F/\partial r (P)=0$ and $F(P)=0$.
We can then easily check that
$$ 0 = g_n(r_0)F_r(P) - g_n'(r_0)F(P) = g_n(t_0)G_n(r_0) $$
and
$$ 0 = g_{n+1}(r_0)F_r(P) - g_{n+1}'(r_0)F(P) = g_{n+1}(t_0)G_n(r_0). $$
If $G_n(r_0)\neq 0$ then $g_n(t_0)=g_{n+1}(t_0)=0$. However $f_n g_n - f_{n-1} g_{n+1}=f_n^2-f_{n-1}f_{n+1}=1$ implies $g_n$ and $g_{n+1}$ are relatively prime polynomials and contradicts $g_n(t_0)=g_{n+1}(t_0)=0$. Thus $G_n(r_0)=0$ and also $G_n(t_0)=0$.
\end{proof}

\begin{lemma}\label{intptscorr}
The affine parts of $X_0$ and $X_1$ are smooth. Furthermore, their affine intersection points correspond to irreducible representations and are determined 2-to-1 with the intersection points of $D_0$ and $D_1$.
\end{lemma}

\begin{proof}
From \autoref{thm:smooth}, $D_0$ and $D_1$ are smooth and isomorphic to $Y_0$ and $Y_1$ outside of the points with $(r-2)f_n(r)=0$. This implies that all the affine points of $Y_0$ and $Y_1$, and equivalently of $X_0$ and $X_1$ outside of the points with $(r-2)f_n(r)=0$ are smooth. The affine part of $Y(\Gamma)$ does not contain points $(r,y)$ satisfying $f_n(r)=0$. Similarly, the affine part of $X(\Gamma)$ does not contain points $(r,x)$ satisfying $f_n(r)=0$. 

Let $F=f_n(t(r,x))\left( f_n(r)g_n(r)(-x^2+2+r) -1 \right)+f_{n-1}(t(r,x))$ with $t(r,x)=(2-r)(x^2-2-r)f_n^2(r)+2$, the defining polynomial for $X$ as in \autoref{X}.
Suppose $(r_0,x_0)$ is a point in the affine part of $X$ with $r_0=2$. Then $t(r_0,x_0)=2$ and $x_0^2=\frac{4 n^2-1}{n^2}$. In particular, $x_0\neq0$.
Let $F_x=\partial F/\partial x$. Then
\begin{align*}
F_x = & \frac{\partial f_n(t)}{\partial t}\frac{dt}{dx} \left( f_n(r)g_n(r)(-x^2+2+r) -1 \right) \\ & + f_n(t)f_n(r)g_n(r)(-2x) + \frac{\partial f_{n-1}(t)}{\partial t}\frac{dt}{dx}
\end{align*}
with
\begin{equation*}
\frac{dt}{dx}=(2-r)f_n^2(r)(x).
\end{equation*}
Evaluating at $(r_0,x_0)$ we get $\frac{dt}{dx}(r_0,x_0)=0$ and
\begin{align*}
F_x(r_0,x_0) & =  0 + f_n(2)f_n(2)g_n(2)(-2x_0) + 0 \\
& =  (n)(n)(n-n+1)(-2x_0) \text{ by \autoref{fhfacts}\ref{fj2}} \\
& = -2n^2 x_0 \\
& \neq 0.
\end{align*}
Therefore $(r_0,x_0)$ is a smooth point in $X$. In particular, $(r_0,x_0)$ is not an intersection point of $X_0$ and $X_1$.
 
If $(r_1,x_1)$ is a point in the affine intersection of $X_0$ and $X_1$ then $r_1\neq 2$, so it corresponds to an irreducible representation. This point $(r_1,x_1)$ maps to a point in the intersection of $D_0$ and $D_1$ via the map $X(\Gamma_n) \rightarrow Y(\Gamma_n) \rightarrow D(\Gamma_n)$ (the covering map composed with the birational equivalence).
\end{proof}

\section{Proof of the main result}\label{sec:proofs}
\autoref{intro thm} will follow from \autoref{main} and \autoref{0slope}.

\subsection{Surface detection} In this section we show that the intersection points detect essential surfaces.

\begin{theorem}\label{main}
Every intersection point in $X(\Gamma_n)$ detects an essential surface in the complement of $J(2n,2n)$ in $S^3$.
\end{theorem}

\begin{proof}
The work of Culler and Shalen (see \cite[Theorem 2.2.1, Proposition 2.3.1]{cullershalen}) shows that any ideal point in the character variety will give rise to an embedded essential surface in the knot exterior. Thus we need only consider the affine intersection points.

By \autoref{intptscorr}, any affine intersection point $(r_0,x_0)$ of $X_0$ and $X_1$ maps to the point $(r_0,r_0)$ in the intersection of $D_0$ and $D_1$, since $D_0$ is defined by the line $r-t$ (see \autoref{thm:smooth}.\ref{line r=t}). Therefore by \autoref{rcoor}, $G_n(r_0)=0$. Notice that the defining polynomials for $D_0$ and $D_1$ have degree $1$ and $2n-2$ respectively, so by Bezout's Theorem for smooth algebraic curves, they have $2n-2$ distinct intersections points. Since $G_n$ is of degree $2n-2$ (see \autoref{fhfacts}\ref{hpoly}), it must be that the roots of $G_n$ exactly determine the intersection points of $D_0$ and $D_1$. We will now show that the $x_0$ at each intersection point $(r_0,x_0)$ are not algebraic integers. It will then follow that the affine intersection points detect essential surfaces (see \autoref{subsec: ani}).

Let $H=\Q(r_0,\mu_0)$ and $v$ a valuation on $H$ with $v(\pi)=1$ and for some uniformiser $\pi$ over the prime 2 and $\mu_0$ an eigenvalue of $A$. Let $\mathfrak{p}$ be the prime associated with $v$ and $\F_\mathfrak{p}$ its residue field. Then $\F_\mathfrak{p}$ has characteristic 2. Combining \ref{fhfacts}\ref{hpol} and \autoref{fhfacts}\ref{f2j} we get 
\begin{equation*}
(u+2)G_n=uf_{n}^2-2f_nf_{n-1}+2n.
\end{equation*}
The reduction of this equation to $\F_\mathfrak{p}$ shows $G_n=f_{n}^2$ over $\F_\mathfrak{p}$. Evaluating at $r_0$  gives $0=f_{n}^2(r_0)$ in $\F_\mathfrak{p}$. Therefore $v(f_{n}^2(r_0))>0$, so $f_{n}^2(r_0)$ is not a unit and thus $\frac{1}{f_{n}^2(r_0)}$ is not an algebraic integer. Combining (\autoref{t}) with $t_0=r_0$ we get that
\begin{equation}\label{xformula}
x_0^2 = 2+r_0-\frac{1}{f_n^2(r_0)}
\end{equation}
is not an algebraic integer.
\end{proof}

\subsection{Detected slope}\label{sec: det slope}
In this section we determine the slope of the detected surfaces and prove the following theorem.

\begin{theorem}\label{0slope}
The affine intersection points in $X(\Gamma_n)$ detect a Seifert surface.
\end{theorem}

The following trace identities for $M_1,M_2\in\SL_2\C$ follow from Cayley-Hamilton:
\begin{equation}\label{trinv}
tr(M_1)=tr(M_1^{-1}) 
\end{equation}
\begin{equation}\label{trorder}
tr(M_1 M_2)=tr(M_2 M_1)
\end{equation}
\begin{align}\label{matrixtrace}
tr (M_1 M_2) & = (tr M_1) (tr M_2) - tr(M_1^{-1}M_2) \\
& = (tr M_1) (tr M_2) - tr(M_1 M_2^{-1}).  \notag
\end{align}

The following identities follows from the previous identities by induction:
\begin{align}\label{trpow}
tr(M_1^k) & = tr(M_1)f_k(tr(M_1)) - 2 f_{k-1}(tr(M_1)) \\
& = f_{k+1}(tr(M_1)) - f_{k-1}(tr(M_1))  \notag 
\end{align}
\begin{align}\label{trcommutator}
tr[M_1,M_2] & = tr(M_1 M_2 M_1^{-1} M_2^{-1}) \\
& = tr^2(M_1) + tr^2(M_2) + tr^2(M_1 M_2) - tr(M_1) tr(M_2) tr(M_1 M_2)-2. \notag
\end{align}

\begin{lemma} \label{trprodSs}
Let $S_1$ and $S_2$ be the images of $s_1$ and $s_2$ at a representation corresponding to a point $(r,x)$ in $X(\Gamma_n)$. The trace of $S_1 S_2^{-1}$ is given by
\begin{equation}
f_n(r)\left( f_n(t)\delta_{1,1}-r f_{n-1}(t)\right) - f_{n-1}(r)\left( f_{n+1}(t)- f_{n-1}(t)\right)
\end{equation}
with $$t=(2-r)(x^2-2-r)f_{n}^2(r)+2,$$
and $$\delta_{1,1}=tr(W_n B A^{-1})=(2-r) f_{n-1}(r)f_n(r)(x^2-2-r)+r . $$
\end{lemma}

\begin{proof}
Recall that $tr(S_1)=tr(W_n^n)$, where $tr(W_n)=t$ and $tr(S_2^{-1})=tr((BA^{-1})^n)$, where $tr(BA^{-1})=tr(AB^{-1})=r$.

Set $\delta_{d,e}=tr(W_n^d (BA^{-1})^e)$ and 
\begin{equation*}
\gamma_{d,e}= f_e(r)\left( f_d(t)\delta_{1,1}-r f_{d-1}(t)\right) - f_{e-1}(r)\left( f_{d+1}(t)- f_{d-1}(t)\right).
\end{equation*}
The statement of the lemma is equivalent to $\delta_{d,e}=\gamma_{d,e}$ in the case $d=n$ and $e=n$. 

We have
\begin{align*}
\delta_{d,0} & = tr(W_n^d) \\
& = tr(W_n) f_d(tr(W_n))-2f_{d-1}(tr(W_n)) \text{ by \autoref{trpow} } \\
& = t f_d(t)-2 f_{d-1}(t) \\
& = f_{d+1}(t)- f_{d-1}(t) \\
& = \gamma_{d,0}
\end{align*}
and
\begin{align*}
\delta_{0,e} & = tr((BA^{-1})^e) \\
& = tr(BA^{-1})f_e(tr(BA^{-1}))-2f_{e-1}(tr(BA^{-1})) \text{ by \autoref{trpow} } \\
& = tr(AB^{-1})f_e(tr(AB^{-1}))-2f_{e-1}(tr(AB^{-1})) \text{ by \autoref{trinv} } \\
& = r f_e(r)-2 f_{e-1}(r) \\
& = \gamma_{0,e}.
\end{align*}

Clearly $\gamma_{1,1}=\delta_{1,1}$, which is given by
\begin{align*}
\delta_{1,1} & = tr(W_n B A^{-1}) \\
& = tr(B A^{-1} W_n) \text{ by \autoref{trorder} } \\
& = tr(B A^{-1} (AB^{-1})^n (A^{-1}B)^n) \\
& = tr((AB^{-1})^{n-1} (A^{-1}B)^{n-1}A^{-1}B) \\
& = tr(W_{n-1}A^{-1}B) \\
& = (2-r) f_{n-1}(r)f_n(r)(x^2-2-r)+r \text{ by \cite[Lemma 3.6]{smooth}}.
\end{align*}

Notice that $\delta_{d,1}$ satisfies the recursion
\begin{align*}
\delta_{d,1} & = tr(W_n^d B A^{-1}) \\
& = tr(W_n) tr(W_n^{d-1} B A^{-1})- tr(W_n^{d-2} B A^{-1})\\
& = t \delta_{d-1,1} -\delta_{d-2,1}
\end{align*}
as does 
\begin{align*}
\gamma_{d,1} & = f_d(t)\delta_{1,1}-r f_{d-1}(t) \\
& = t f_{d-1}(t)\delta_{1,1}-f_{d-2}(t)\delta_{1,1}-r t f_{d-2}(t)+ r f_{d-3}(t) \\
& = t \left( f_{d-1}(t)\delta_{1,1} -r f_{d-2}(t)\right)- \left( f_{d-2}(t)\delta_{1,1}+ r f_{d-3}(t)\right) \\
& = t \gamma_{1,1} -\gamma_{2,1}.
\end{align*}

Also notice that $\delta_{d,e}$ satisfies the recursion
\begin{align*}
\delta_{d,e} & = tr(W_n^d (B A^{-1})^e) \\
& = tr(W_n^d (B A^{-1})^{e-1}B A^{-1})\\
& = tr(W_n^d (B A^{-1})^{e-1})tr(B A^{-1})-tr(W_n^d (B A^{-1})^{e-2}) \\
& = r \delta_{d,e-1}-\delta_{d,e-2}.
\end{align*}
as does 
\begin{align*}
\gamma_{d,e} & = f_e(r)\left( f_d(t)\delta_{1,1}-r f_{d-1}(t)\right) - f_{e-1}(r)\left( f_{d+1}(t)- f_{d-1}(t)\right)\\
& = (r f_{e-1}(r) - f_{e-2}(r))\left( f_d(t)\delta_{1,1}-r f_{d-1}(t)\right) \\ 
& \hspace{72pt} - (r f_{e-2}(r)-f_{e-3}(r))\left( f_{d+1}(t)- f_{d-1}(t)\right) \\
& = r \gamma_{d,e-1} -\gamma_{d,e-2}.
\end{align*}
Since the equivalence is satisfied for $\delta_{0,0}$, $\delta_{1,0}$, $\delta_{0,1}$, $\delta_{1,1}$, this completes the proof. 
\end{proof}

\begin{lemma}\label{lem: slope 0 is Seifert}
If $S$ is a connected essential surface in the exterior of $J(2n,2n)$ with slope zero, then $S$ is a genus 1 Seifert surface. 
\end{lemma}

\begin{proof}
Recall from \autoref{subsec:family} that the knots $J(2n,2n)$ have two-bridge normal form $(4n^2-1,2n)$.
Using the language of \cite{hatcherthurston}, the unique continued fraction expansion for the knot $J(2n,2n)=K_{2n/(4n^2-1)}$ of the form $[a_1,-a_2,a_3,\dots,\pm a_k]$ as in \cite[Figure 5]{hatcherthurston} and the proceeding paragraph is given by $[2n-1,-1,2n-1]$. By \cite[Theorem 1(c)]{hatcherthurston} and the remarks on \cite[page 229]{hatcherthurston} and the top of \cite[page 230]{hatcherthurston}, any essential surface is carried by a branched surface corresponding to a minimal edge path involving only the heavy lines in \cite[Figure 5]{hatcherthurston}. Following the remarks at the end of \cite[page 229]{hatcherthurston}, there are four minimal edge paths. These correspond to the following continued fraction expansions:
\begin{equation*}
[\underbrace{-2,\dots,-2}_{2n-2},-3,\underbrace{-2,\dots,-2}_{2n-2}],
[\underbrace{-2,\dots,-2}_{2n-1},2n-1],
[2n-1,\underbrace{-2,\dots,-2}_{2n-1}],
[2n,2n].
\end{equation*}
By \cite[Proposition 2]{hatcherthurston}, the branched surfaces associated to these continued fraction expansions will carry essential surfaces of slopes determined solely by the continued fraction expansion. The corresponding slopes are $2-8n$, $-4n$, $-4n$ and $0$ respectively.

Any connected surface of slope zero is therefore carried by the branched surface $\Sigma[2n,2n]$. By \cite[Proposition 1(1)]{hatcherthurston} and the remark directly following it, this surface is a single-sheeted orientable surface. Such an essential connected surface of slope zero is then isotopic to either $S_1(0)$ or $S_1(1)$ as constructed in \cite[page 227]{hatcherthurston}. It is easy to see from the construction, that $S_1(0)$ and $S_1(1)$ are non-separating surfaces. Therefore, a connected essential surface of slope zero in the exterior of $J(2n,2n)$ is a Seifert surface.

By \cite[Corollary to Proposition 1]{hatcherthurston}, all essential Seifert surfaces for a two-bridge knot have the same genus. Since the Seifert surface described in \autoref{subsec:family} has genus 1, all essential Seifert surfaces for $J(2n,2n)$ also have genus 1.
\end{proof}

We can now prove \autoref{0slope}.

\begin{proof}[Proof of \autoref{0slope}]
Suppose that $(r_0,x_0)$ is a point in the affine intersection of $X_0$ and $X_1$ and recall that $t_0=r_0$ (see \autoref{thm:smooth}.\ref{line r=t}). Let $S_1$ and $S_2$ be the images of $s_1$ and $s_2$ of a representation corresponding to $(r_0,x_0)$. Recall from \autoref{subsec:family} that the preferred longitude, is given by $(W_n)^n(W_n^*)^n=S_1S_2^{-1}S_1^{-1}S_2$.

By \autoref{trcommutator}, $S_1S_2^{-1}S_1^{-1}S_2$ has trace
\begin{align*}
tr(S_1S_2^{-1}S_1^{-1}S_2) & = tr^2(S_1) + tr^2(S_2) + tr^2(S_1 S_2^{-1}) - tr(S_1) tr(S_2) tr(S_1 S_2^{-1})-2\\
& = t_0^2 + r_0^2 + tr^2(S_1 S_2^{-1}) - t_0 r_0 tr(S_1 S_2^{-1})-2\\
& = 2 r_0^2 + (1-r_0^2)tr^2(S_1 S_2^{-1})-2.
\end{align*}

Since $tr(W_n)=t_0=r_0$, $S_1=W_n^n$ has trace $r_0 f_n(r_0) - 2 f_{n-1}(r_0)$. Since $tr(AB^{-1})=r_0$, also $S_2=(AB^{-1})^n$ has trace $r_0 f_n(r_0) - 2 f_{n-1}(r_0)$. From \autoref{trprodSs}, $S_1 S_2^{-1}$ has trace
$$ (2-r_0)(x_0^2-2-r_0)f_{n-1}f_n^3+r_0f_n^2-r_0 f_{n-1}f_n- f_{n-1}f_{n+1}+ f_{n-1}^2 $$
evaluated at $r_0$.
Since $r_0$ is an algebraic integer (see \autoref{rcoor}, \autoref{fhfacts}.\ref{hpoly}), it suffices to show that $x_0^2f_n^2(r_0)$ is an algebraic integer, guaranteeing integrality of $tr(S_1 S_2^{-1})$. Applying \autoref{xformula} we get
\begin{align*}
x_0^2f_n^2(r_0) & = \left(2+r_0-\frac{1}{f_n^2(r_0)}\right) f_n^2(r_0) \\
& =  (2+r_0)f_n^2(r_0)-1.
\end{align*}
which is an algebraic integer. The theorem now follows from \autoref{lem: schanuelzhang} and \autoref{lem: slope 0 is Seifert}.
\end{proof}

\begin{remark}
There are finitely many characters of reducible representations in $X(\Gamma_n)$. These are contained in $X_0$ and also detect the slope zero. To see this, let $(r_0,x_0)\in X$ correspond to a reducible representation $\rho$. Then $r_0=2$. Substituting $r_0=2$ at \autoref{t} we get $t_0=2$ and at \autoref{X} we get
\begin{equation}\label{xformula2}
x_0^2 =\frac{4 n^2-1}{n^2}
\end{equation}
which is not an algebraic integer.  We may conjugate $\rho$ so that $\rho(\Gamma)$ is generated by
\begin{equation*}
\begin{pmatrix} \mu & 1 \\ 0 & \mu^{-1} \end{pmatrix}
 \hspace{8pt}\text{and}\hspace{8pt}
\begin{pmatrix} \mu & 0 \\ 0 & \mu^{-1} \end{pmatrix}.
\end{equation*}
Then the character $\rho$ is the same as a character of a diagonal representation, which is abelian. Therefore the traces of the images of $s_1$, $s_2$ and $s_1 s_2^{-1}$ are all the same as the trace of the image of the identity, which is the integer 2.
\end{remark}

\section{Two examples}\label{sec:examples}
We consider in detail the first two knots in the family $J(2n,2n)$, namely $7_4$ ($n=2$) and $11a_{363}$ ($n=3$).

\subsection{The First Knot}
The first knot in the family $J(2n,2n)$ is the knot $7_4$ of two-bridge normal form $(15,11)$ with knot group
\begin{equation*}
\Gamma_2=\langle a,b : a w^2 = w^2 b \rangle
\end{equation*}
where $w=ab^{-1}ab^{-1}a^{-1}ba^{-1}b$.
The variety $X(\Gamma_2)$ is defined by the polynomial
\begin{equation*}
(-1 + 2 r^2 + r^3 - r^2 x^2) (1 + 4 r - 4 r^2 - r^3 + r^4 - 2 r x^2 + 
   3 r^2 x^2 - r^3 x^2)
\end{equation*}
where the first factor defines the canonical component $X_0$ and the second factor defines the component $X_1$. These two curves intersect at 20 points counting multiplicities. However, 16 of these correspond to 2 ideal points (each with multiplicity 8). The affine intersections points $(r,x)$ are
\begin{align*}
\left(1-i,\sqrt{3-\frac{3i}{2}}\right),&
\left(1-i,-\sqrt{3-\frac{3i}{2}}\right),\\
\left(1+i,\sqrt{3+\frac{3i}{2}}\right),&
\left(1+i,-\sqrt{3+\frac{3i}{2}}\right)
\end{align*}
each with multiplicity 1.
The $x$-coordinates of these points are the 4 roots of the polynomial $4x^4-24x^2+45$. These algebraic non-integral numbers determine the traces of the meridian.

Consider the representation $\rho:\Gamma_2\rightarrow \SL_2\C$ given by
\begin{equation*}
 \rho(a)= \begin{pmatrix} \mu & 1 \\ 0 & \mu^{-1} \end{pmatrix}
 \hspace{8pt}\text{and}\hspace{8pt}
 \rho(b)= \begin{pmatrix} \mu & 0 \\ 1+i & \mu^{-1} \end{pmatrix}
\end{equation*}
corresponding to the point $\left(1-i,\sqrt{3-\frac{3i}{2}}\right)$,
with $\mu=\frac{1}{2} \left(\sqrt{-1 - \frac{3i}{2}} + \sqrt{3 - \frac{3i}{2}}\right)$. The image of the longitude is the matrix
\begin{equation*}
\begin{pmatrix} 7+12i+2\sqrt{-24+42i} & -8\sqrt{-3-6i} \\ 0 & 7+12i-2\sqrt{-24 + 42i} \end{pmatrix}
\end{equation*}
with trace $14+24i$, an algebraic integer.

\begin{remark} It is easy to see that for the representations given by these affine intersection points, the restriction of the peripheral subgroup is faithful. The meridian and longitude are mapped to loxodromics with the same axis. However, since one has non-integral trace and the other integral trace, these generate a non-discrete $\Z^2$. This leads to ask the following question. Could these representations be faithful? A positive answer would imply that the non-canonical component contains faithful representations, and in particular do not come from a quotient. 
\end{remark}

\subsection{The Second Knot}
The second knot in the family $J(2n,2n)$ is the knot $11a_{363}$ of two-bridge normal form $(35,29)$ with knot group
\begin{equation*}
\Gamma_2=\langle a,b : a w^3 = w^3 b \rangle
\end{equation*}
where $w=ab^{-1}ab^{-1}ab^{-1}a^{-1}ba^{-1}ba^{-1}b$.
The variety $X(\Gamma_3)$ is defined by the polynomial
$$ (1+r-4r^2-2r^3+2r^4+r^5-x^2+2r^2x^2-r^4x^2)(1 + 8 r - 40 r^2 - 46 r^3 + 110 r^4$$
$$+ 71 r^5 - 113 r^6 - 43r^7+54r^8+11r^9-12r^{10}-r^{11}+r^{12}-8x^2 -8rx^2+ 60r^2x^2 $$
$$+ 21r^3x^2 - 130r^4x^2- 7r^5x^2+118 r^6 x^2 - 16 r^7 x^2 - 46 r^8 x^2 + 12 r^9 x^2 + 6 r^{10} x^2 $$
$$-2 r^{11} x^2 + 4 x^4 - 19 r^2 x^4 + 5 r^3 x^4 + 32 r^4 x^4 
- 15 r^5 x^4 - 22 r^6 x^4
+ 15 r^7 x^4 + 4 r^8 x^4 $$
$$-5r^9x^4+r^{10}x^4)$$
where the first factor defines the canonical component $X_0$ and the second factor defines the component $X_1$. These two curves intersect at 84 points counting multiplicities. However, 76 of these correspond to 2 ideal points (with multiplicities 24 and 52). There is 8 affine intersections points $(r,x)$ each with multiplicity 1. The $r$-coordinates are the four roots of the polynomial $r^4-2r^3+3$. The $x$-coordinates are the eight roots of the polynomial $144x^8-1424x^6+5160x^4 -8400x^2+6125$. These algebraic non-integral numbers determine the traces of the meridian.

Consider the representation $\rho:\Gamma_3\rightarrow \SL_2\C$ given by
\begin{equation*}
 \rho(a)= \begin{pmatrix} \mu & 1 \\ 0 & \mu^{-1} \end{pmatrix}
 \hspace{8pt}\text{and}\hspace{8pt}
 \rho(b)= \begin{pmatrix} \mu & 0 \\ s & \mu^{-1} \end{pmatrix}
\end{equation*}
corresponding to one the intersection points, with $\mu\approx 0.44228 + 0.601587i$ (an algebraic number of degree 8 over $\Q$) and $s\approx 2.60504 + 0.835079 i$ (an algebraic integral of degree 4 over $\Q$).
The image of the longitude has trace a root of the polynomial
\begin{equation*}
\ell^4-212\ell^3+15768\ell^2-385360\ell+ 8647328
\end{equation*}
($\approx 95.247 + 42.4755i$), an algebraic integer.

\section{Final Remarks}\label{sec:final}

\subsection{Multiple Components}\label{subsec: final mult components}
In \cite{riley}, Riley describes 3 cases in which a non-canonical component of characters of irreducible representations can arise in the character variety of two-bridge knots. One way we get a non-canonical component is if there exists an epimorphism from the knot group onto another knot group. However, this is not the case for the knots $J(2n,2n)$. 

\begin{claim}\label{claim no epi}
There is no epimorphism from $\Gamma_n$ onto another knot group.
\end{claim}

\begin{proof}
The knot $J(2n,2n)$ has Alexander polynomial
\begin{equation*}
\Delta_n(t)=n^2 t^2 + (1-2n^2) t +n^2.
\end{equation*}
Since its quadratic discriminant, $1-4n^2$, is negative, $\Delta_n$ is an irreducible integral polynomial.

Denote the knot $J(2n,2n)$ by $K$ and suppose there exists an epimorphism from $\Gamma_n$ onto the knot group $\Gamma'$ for some other knot $K'$. The Alexander polynomial of $K'$ must divide $\Delta_n$ (see e.g. Remark (3) of \cite[Proposition 1.11]{simon}) and furthermore, $K'$ is necessarily a two-bridge knot \cite[Corollary 1.3]{simon}. However, two-bridge knots have nontrivial Alexander polynomials. Therefore $K'$ must have the same Alexander polynomial $\Delta_n(t)$.

Let $\widetilde{M}$ and $\widetilde{M'}$ denote the infinite cyclic covers of $S^3-J(2n,2n)$ and $S^3-K'$ respectively.
In \cite{mayland}, Mayland expressed the derived groups $\gamma(M)$ and $\gamma(M')$ of $\widetilde{M}$ and $\widetilde{M'}$ for any two-bridge knots as a union of parafree groups in such a way that \cite[Proposition 2.1]{baum} applies to show $\gamma(M)$ and $\gamma(M')$ are residually torsion-free nilpotent. That is, $\gamma(M)_\omega\cong 1\cong\gamma(M')_\omega$ where $G_\omega$ is the $\omega$-term in the lower central series and $\omega$ is the first infinite cardinal.

Since the knots share the same Alexander polynomial, $H_1(\widetilde{M})\cong H_1(\widetilde{M'})$. We can now apply a theorem of Stallings \cite[Theorem 3,4]{stallings} to the epimorphism $h:\pi_1(\widetilde{M}) \rightarrow \pi_1(\widetilde{M'})$ to conclude $h$ is an ismorphism. Therefore, $\Gamma_n$ and $\Gamma'$ are isomorphic.
\end{proof}

Note that \autoref{claim no epi} was also proved in \cite[Proposition 3.1]{minimality}.

Another way in which non-canonical components of characters of irreducible representations can arise in the character variety is when the knot has a certain nice symmetry described by Ohtsuki. In particular, whenever a two-bridge knot has two-bridge normal form $(\alpha,\beta)$ with $\beta^2\equiv 1\mod\alpha$ and $\beta\neq 1$, there is a diagram from which one can see an orientation preserving involution. This involution induces a nontrivial action on the character variety. However, it fixes a neighborhood of the character of a holonomy representation. Therefore, there exists a non-canonical component containing characters of irreducible representations (see \cite[Proposition 5.5]{ohtsuki}).
Notice that the knots studied in this paper satisfy these conditions. They have two-bridge normal form $(4n^2-1,4n^2-2n-1)$.

\subsection{Other Examples of Two-Bridge Knots with Two Components}
In \autoref{table: 2-comp knots} we list 
knots with crossing number at most 9 whose character variety contains exactly two components of irreducible components. For all of these, the intersection points are Galois conjugates and detect the same slope. The table includes the 2-bridge normal form, the detected slope, whether or not the knot is fibered or has the $(p,q)$-symmetry described in \autoref{subsec: final mult components}, and if there is an epimorphism from the knot group to another knot group. Whenever a knot is fibered, a Seifert surface cannot be detected by ideal nor by algebraic non-integral points in the character variety.

\begin{table}
\begin{tabular}{ |c|c|c|c|c|c| } 
 \hline
 knot & $(p,q)$ & detected slope & fibered & $(p,q)$-symmetry & epimorphism \\
 \hline
	$7_{4}$ & $(15,11)$ & 0 & no & yes & no \\  
	$7_{7}$ & $(21,13)$ & 6 & yes & yes & no\\ 
	$8_{11}$ & $(27,19)$ & 6 & no & no & no \\ 
	$9_{6}$ & $(27,5)$ & 18 & no & no & $3_1$ \\ 
	$9_{17}$ & $(39,25)$ & 10 & yes & yes & no \\
 \hline
\end{tabular}
\caption{Knots with two components of irreducible representations} \label{table: 2-comp knots}
\end{table}

In addition to these, the knot groups for the knots $10_{5}$, $10_{9}$ and $10_{32}$ are known to have epimorphisms onto the trefoil knot group. Indeed, the two-bridge knots $9_{6}$, $10_{5}$, $10_{9}$ and $10_{32}$ are the only knots up to 10 crossings whose knot groups have epimorphisms to another two-bridge knot (see \cite[Theorem 1.1]{ordering}). The knot groups surject to the trefoil knot group such that the peripheral subgroup is sent to the peripheral subgroup of the trefoil knot group. Since the non-canonical component of the character variety corresponds to the canonical component of the trefoil character variety, the detected slopes correspond to detected slopes of the trefoil knot. As a fibered knot, the only detected slope of the trefoil knot is 6, so the detected slopes for $9_{6}$, $10_{5}$, $10_{9}$ and $10_{32}$ are multiples of 6.

\subsection{Two-Bridge Knots with Three Components}
One may also want to consider two-bridge knots with three distinct components of irreducible representations in the character variety. Two examples of these are the knots $9_{23}$ and $10_{40}$ with two-bridge normal forms $(45,19)$ and $(75,29)$ both of which satisfy the symmetry condition described above and provide epimorphisms to the trefoil group. These two knots each have character varieties with a canonical component, a distinct component corresponding to the symmetry condition, and a distinct component corresponding to the canonical component of the character variety of the trefoil (to see that the knot group has an epimorphism to the trefoil knot group, refer to \cite[Theorem 1.1]{ordering}. All pairwise intersection points between these three components are algebraic non-integral with the trace of the meridian non-integral by a prime over 2, and correspond to irreducible representations. We note that character variety of the knot $10_{40}$ has triple intersection points between these three components.

\bibliographystyle{plain}

\begin{thebibliography}{888}

\bibitem{baum} Baumslag, Gilbert. Groups with the same lower central sequence as a relatively free group. II. Properties. {\it Trans. Amer. Math. Soc.} 142 (1969) 507-538.

\bibitem{simon} Boileau, Michel; Boyer, Steve; Reid, Alan W.; Wang, Shicheng. Simon's conjecture for two-bridge knots. {\it Comm. Anal. Geom.} 18 (2010), no. 1, 121-143. 
 
\bibitem{seminorms} Boyer, Steve; Zhang, Xingru. On Culler-Shalen seminorms and Dehn filling. {\it Ann. of Math.} (2)  148  (1998),  no. 3, 737-801. 

\bibitem{cgls} Culler, Marc; Gordon, Cameron; Luecke, John; Shalen, Peter B. Dehn surgery on knots. {\it Ann. of Math.} (2) 125 (1987), no. 2, 237-300. 

\bibitem{cullershalen} Culler, Marc; Shalen, Peter B.
Varieties of group representations and splittings of 3-manifolds. {\it Ann. of Math.} (2) 117 (1983), no. 1, 109-146. 
 
\bibitem{hatcherthurston} Hatcher, Allen; Thurston, William. Incompressible surfaces in 2-bridge knot complements. {\it Invent. Math.} 79 (1985), no. 2, 225-246.
  
\bibitem{apoly} Hoste, Jim; Shanahan, Patrick D.  A formula for the A-polynomial of twist knots. {\it J. Knot Theory Ramifications} 13 (2004), no. 2, 193-209. 
 
\bibitem{ordering} Kitano, Teruaki; Suzuki, Masaaki. A partial order in the knot table. {\it Exp. Math.} 14 (2005), no. 4, 385–390.

\bibitem{longreid}  Long, Darren D.; Reid, Alan W. Commensurability and the character variety. {\it Math. Res. Lett.}  6  (1999),  no. 5-6, 581-591.

\bibitem{smooth} Macasieb, Melissa L.; Petersen, Kathleen L.; van Luijk, Ronald M. On character varieties of two-bridge knot groups. {\it Proc. Lond. Math. Soc.} (3) 103 (2011), no. 3, 473-507.

\bibitem{mayland} Mayland, Edward J., Jr. Two-bridge knots have residually finite groups. {\it Proceedings of the Second International Conference on the Theory of Groups (Australian Nat. Univ., Canberra, 1973)}, pp.488-493. Lecture Notes in Math., Vol. 372, Springer, Berlin, 1974. 

\bibitem{minimality} Nagasato, Fumikazu; Suzuki, Masaaki; Tran, Anh T. On minimality of two-bridge knots. {\it Internat. J. Math.} 28 (2017), no. 3, 1750020.

\bibitem{stallings}Stallings, John. Homology and central series of groups. {\it J. Algebra} 2 1965 170-181. 
  
\bibitem{ohtsuki} Ohtsuki, Tomotada. Ideal points and incompressible surfaces in two-bridge knot complements. {\it J. Math. Soc. Japan} 46 (1994), no. 1, 51-87.

\bibitem{riley} Riley, Robert. Algebra for Heckoid groups. {\it Trans. Amer. Math. Soc.} 334 (1992), no. 1, 389-409. 

\bibitem{riley2} Riley, Robert. Nonabelian representations of 2-bridge knot groups. {\it Quart. J. Math. Oxford Ser.} (2) 35 (1984), no. 138, 191-208. 
 
\bibitem{riley3} Riley, Robert. Parabolic representations of knot groups. I. {\it Proc. London Math. Soc.} (3)  24 (1972), 217-242.
 
\bibitem{schanuelzhang} Schanuel, Stephen H.; Zhang, Xingru. Detection of essential surfaces in 3-manifolds with $\SL_2$-trees. {\it Math. Ann.} 320 (2001), no. 1, 149-165. 
 
\bibitem{handbook} Shalen, Peter B. Representations of 3-manifold groups. {\it Handbook of geometric topology}, 955-1044, North-Holland, Amsterdam, 2002. 

\bibitem{serre} Serre, Jean-Pierre. {\it Trees.} Springer Monographs in Mathematics. Springer-Verlag, Berlin, 2003.

\end{thebibliography}

\end{document}